\documentclass[final,onefignum,onetabnum]{siamart190516}

\usepackage{amsmath}
\usepackage{amssymb}
\usepackage{cases}
\usepackage{caption}
\usepackage{subcaption}

\usepackage{nicefrac}

    \usepackage[shortlabels]{enumitem}

    \newlist{UR}{enumerate}{1}
    \setlist[UR]{label=(H\arabic*)}

\newcommand{\dom}{D}
\newcommand{\ran}{R}

\newcommand{\R}{\mathbb{R}}

\usepackage{lipsum}
\usepackage{amsfonts}
\usepackage{graphicx}
\usepackage{epstopdf}
\usepackage{algorithmic}
\ifpdf
  \DeclareGraphicsExtensions{.eps,.pdf,.png,.jpg}
\else
  \DeclareGraphicsExtensions{.eps}
\fi

\newsiamremark{assumption}{Assumption}
\crefname{assumption}{Assumption}{Assumptions}
\newsiamremark{remark}{Remark}
\crefname{remark}{Remark}{Remarks}
\newsiamremark{example}{Example}
\crefname{example}{Example}{Examples}

\headers{Convergence Analysis of the Nonoverlapping Robin--Robin Method}{E.\ Engstr\"om, E.\ Hansen}

\title{Convergence Analysis of the Nonoverlapping Robin--Robin Method for Nonlinear Elliptic Equations\thanks{Version of  \today.
\funding{This work was supported by the Swedish Research Council under the grant 2019--
05396.}}}

\author{Emil Engstr\"om\thanks{Centre for Mathematical Sciences, Lund University,
P.O.\ Box 118, SE-22100 Lund, Sweden, (\email{emil.engstrom@math.lth.se}).}
\and Eskil Hansen\thanks{Centre for Mathematical Sciences, Lund University,
P.O.\ Box 118, SE-22100 Lund, Sweden, (\email{eskil.hansen@math.lth.se}).}}

\usepackage{amsopn}

\ifpdf
\hypersetup{
  pdftitle={Convergence Analysis of the Nonoverlapping Robin--Robin Method for Nonlinear Elliptic Equations},
  pdfauthor={E.\ Engstr\"om, E.\ Hansen}
}
\fi

\begin{document}

\maketitle

\begin{abstract}
We prove convergence for the nonoverlapping Robin--Robin method applied to nonlinear elliptic equations with a $p$-structure, including degenerate diffusion equations governed by the $p$-Laplacian. This nonoverlapping domain decomposition is commonly encountered when discretizing elliptic equations, as it enables the usage of parallel and distributed hardware. Convergence has been derived in various linear contexts, but little has been proven for nonlinear equations. Hence, we develop a new theory for nonlinear Steklov--Poincar\'e operators based on the $p$-structure and the $L^p$-generalization of the Lions--Magenes spaces. This framework allows the reformulation of the Robin--Robin method into a Peaceman--Rachford splitting on the interfaces of the subdomains, and the convergence analysis then follows by employing elements of the abstract theory for monotone operators. The analysis is performed on Lipschitz domains and without restrictive regularity assumptions on the solutions. 
\end{abstract}



\section{Introduction}  Approximating the solution of an elliptic partial differential equation (PDE) typically demands large-scale computations requiring the usage of parallel and distributed hardware. In this context, a nonoverlapping domain decomposition method is a suitable choice, as it can be implemented in parallel with local communication. After decomposing the equation’s spatial domain into nonoverlapping subdomains, the method consists of an iterative procedure that solves the equation on each subdomain and thereafter communicates the results via the boundaries to the adjacent subdomains. For a general introduction we refer to \cite{quarteroni,Widlund}. 

There is a vast amount of methods in the literature, employing different transmission conditions between the subdomains. The standard examples are based on the alternate use of Dirichlet and Neumann boundary conditions, but a competitive alternative is the Robin--Robin method, where the same type of Robin boundary condition is used for all subdomains. 
The Robin--Robin method was introduced in~\cite{lions3} together with a convergence proof when applied to linear elliptic equations. After applying a finite element discretization, convergence rates of the form $1-\mathcal{O}(\sqrt{h})$, with $h$ denoting the mesh width, have been derived in various linear contexts~\cite{GuoHou03,Lui09,XuQin10}; also see~\cite{Gander01}. For generalizations and further applications of the Robin--Robin method applied to linear PDEs we refer to~\cite{Chen14,GuoHou03} and references therein. 

When considering nonlinear elliptic PDEs the literature is more limited. Convergence studies relating to overlapping Schwarz methods are given in~\cite{dryja97,tai98,tai02}. However, there are hardly any results dealing with nonoverlapping domain decomposition schemes. One exception is~\cite{berninger11}, where the convergence of the Dirichlet--Neumann and Robin--Robin methods are analyzed for a family of one-dimensional elliptic equations. A related study is~\cite{schreiber}, where the equivalence between a class of nonlinear elliptic equations and the corresponding transmission problems are proven for nonoverlapping decompositions with cross points, but no numerical scheme is considered. Apart from~\cite{tai02}, all these nonlinear studies rely on frameworks similar to the linear case, e.g., assuming  that the diffusion is uniformly positive.

Hence, the aim of this paper is to derive a genuinely nonlinear extension of the linear convergence result given in~\cite{lions3} for the nonoverlapping Robin--Robin method. We will focus on nonlinear elliptic equations of the form
\begin{equation}\label{eq:class}
\left\{
     \begin{aligned}
            -\nabla\cdot\alpha(\nabla u)+g(u)&=f  & &\text{in }\Omega,\\
             u&=0 & &\text{on }\partial\Omega,
        \end{aligned}
\right.
\end{equation}
where $\Omega$ is a bounded domain in $\R^d$, $d=1,2,\ldots,$ with boundary $\partial\Omega$. The functions $\alpha$ and $g$  are assumed to have a $p$-structure; defined in \cref{sec:pstruct}. This structure enables a clear-cut convergence analysis for a broad family of degenerate elliptic equations, i.e., $\alpha(\nabla u)$ may vanish for nonzero values of $u$. The latter typically prevents the existence of a strong solution in $W^{2,p}(\Omega)$. 

The archetypical examples of nonlinear elliptic equations with a $p$-structure are those 
governed by the $p$-Laplacian, where $\alpha(z)=|z|^{p-2}z$. Examples include the computation of the nonlinear resolvent 
\begin{equation}\label{eq:resolvent}
      -\nabla\cdot(|\nabla u|^{p-2}\nabla u)+\lambda u=f,
\end{equation}
arising in the context of an implicit Euler discretization of the parabolic $p$-Laplace equation, and the nonlinear reaction-diffusion problem 
\begin{equation}\label{eq:eig}
        -\nabla\cdot(|\nabla u|^{p-2}\nabla u)+\lambda |u|^{p-2} u=f.
\end{equation}

For sake of simplicity, we decompose the original domain $\Omega$ into two nonoverlapping subdomains $\{\Omega_i,\text{ for }i=1,2\}$, with boundaries denoted by $\partial\Omega_i$, and separated by the interface $\Gamma$, i.e., 
\begin{displaymath}
\overline{\Omega}=\overline{\Omega}_1\cup\overline{\Omega}_2,\quad \Omega_1\cap\Omega_2=\emptyset\quad\text{and}\quad\Gamma=(\partial\Omega_1\cap\partial\Omega_2)\setminus\partial\Omega. 
\end{displaymath}
Two examples of such decompositions are illustrated in~\cref{fig:sub1,fig:sub2}, respectively. The analysis presented here can also, in a trivial fashion, be extended to the case when $\Omega_i$ is a family of nonadjacent subdomains, e.g., the stripewise domain decomposition illustrated in~\cref{fig:sub3}.
\begin{figure}
\centering
\begin{subfigure}{.4\textwidth}
  \centering
  \includegraphics[width=.7\linewidth]{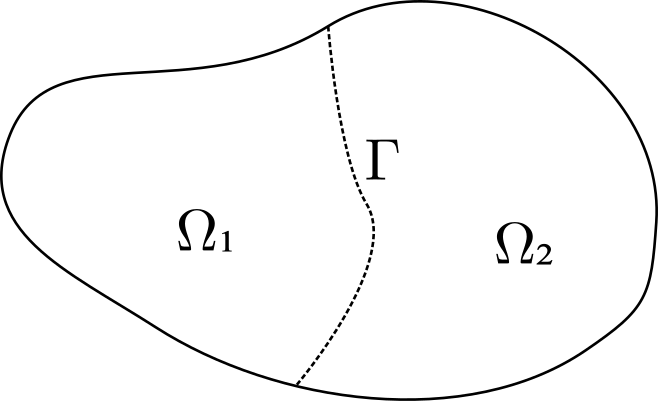}
  \caption{}
  \label{fig:sub1}
\end{subfigure}%
\begin{subfigure}{.4\textwidth}
  \centering
  \includegraphics[width=.7\linewidth]{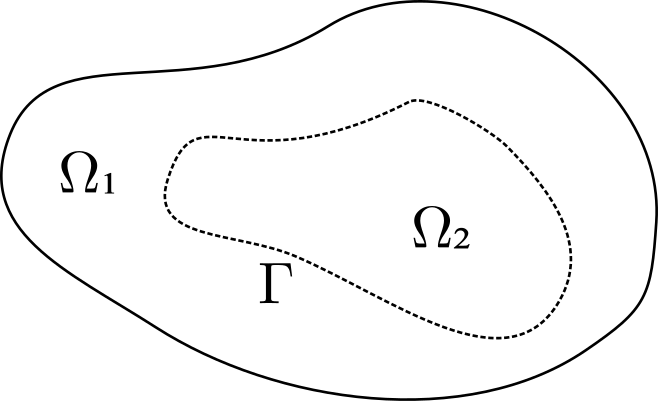}
  \caption{}
  \label{fig:sub2}
\end{subfigure}
\begin{subfigure}{.4\textwidth}
  \centering
  \includegraphics[width=.7\linewidth]{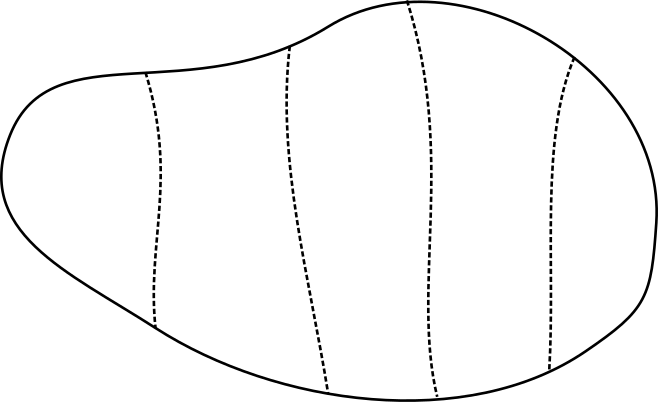}
  \caption{}
  \label{fig:sub3}
\end{subfigure}%
\begin{subfigure}{.4\textwidth}
  \centering
  \includegraphics[width=.7\linewidth]{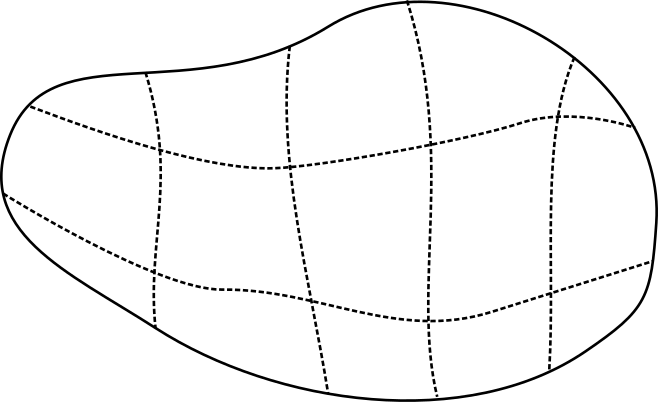}
  \caption{}
  \label{fig:sub4}
\end{subfigure}
\caption{Examples of different domain decompositions: (a) illustrates a domain decomposition with two intersection points; (b) a decomposition without intersection points; (c) a stripewise decomposition without cross points; (d) a full decomposition with cross points.}
\label{fig:test}
\end{figure}
The strong form of the Robin--Robin method applied to \cref{eq:class} is then to find $(u^n_1,u^n_2)$, for $n=1,2,\ldots,$ such that
\begin{equation}\label{eq:RobinStrong}
\left\{
     \begin{aligned}
            -\nabla\cdot\alpha(\nabla u^{n+1}_1)+g(u^{n+1}_1)&=f  & &\text{in }\Omega_1,\\
             u^{n+1}_1&=0 & &\text{on }\partial\Omega_1\setminus\Gamma,\\
             \alpha(\nabla u^{n+1}_1)\cdot\nu_1 + s u^{n+1}_1 &=  \alpha(\nabla u^n_2)\cdot\nu_1 + s u^n_2 & &\text{on }\Gamma,\\[5pt]
 	     -\nabla\cdot\alpha(\nabla u^{n+1}_2)+g(u^{n+1}_2)&=f  & &\text{in }\Omega_2,\\
             u^{n+1}_2&=0 & &\text{on }\partial\Omega_2\setminus\Gamma,\\
             \alpha(\nabla u^{n+1}_2)\cdot\nu_2 + s u^{n+1}_2 &=  \alpha(\nabla u^{n+1}_1)\cdot\nu_2 + s u^{n+1}_1 & &\text{on }\Gamma,
        \end{aligned}
\right.
\end{equation}
where $\nu_i$ denotes the unit outward normal vector of $\partial\Omega_i$, $u_2^0$ is a given initial guess and $s>0$ is a method parameter. 
Here, $u^n_i$ and $u^n_i|_{\Gamma}$ approximate the solution $u$ restricted to $\Omega_i$ and $\Gamma$, respectively. 

The convergence analysis is organized as follows. For linear elliptic equations, i.e., equations with a $2$-structure, the analysis relies on the trace operator from $H^1(\Omega_i)$ onto $H^{1/2}(\partial\Omega_i)$, and the Lions--Magenes spaces $H^{1/2}_{00}(\Gamma)$. We therefore start by introducing the generalized $p$-version of the trace operator, now given from $W^{1,p}(\Omega_i)$ onto $W^{1-1/p,p}(\partial\Omega_i)$, and the corresponding Lions--Magenes spaces $\Lambda_i$; see~\cref{sec:functionspaces,sec:LM}. There is a surprising lack of proofs in the literature dealing with this generalized $p$-setting. We therefore make an effort to give precise definitions and proof references. 

With the correct function spaces in place, we prove that the nonlinear transmission problem is equivalent to the weak form of~\cref{eq:class} in~\cref{thm:TPequivWE}, and introduce the new nonlinear Steklov--Poincar\'e operators, as maps from $\Lambda_i$ to $\Lambda_i^*$, in~\cref{sec:interface}. The latter yields that the transmission problem can be stated as a problem on $\Gamma$ and the Robin--Robin method reduces to the Peaceman--Rachford splitting. The main challenge is then to derive the fundamental properties of the nonlinear Steklov--Poincar\'e operators from the $p$-structure, which is achieved in~\cref{sec:SP}. By interpreting the nonlinear Steklov--Poincar\'e operators as unbounded, monotone maps on $L^2(\Gamma)$ we finally prove that the Robin--Robin method is well defined on~$W^{1,p}(\Omega_1)\times W^{1,p}(\Omega_2)$; see~\cref{cor:PRisRobin}, and convergent in the same space; see \cref{thm:conv}. The latter relies on the abstract theory of the Peaceman--Rachford splittings~\cite{lionsmercier}.

The continuous convergence analysis presented here also holds in the finite dimensional case obtained after a suitable spatial discretization, e.g., by employing finite elements. However, we will limit ourselves to the continuous case in this paper. Hence, important issues including convergence rates for the finite-dimensional case and the influence of the mesh width on the optimal choice of the method parameter $s$ will be explored elsewhere. 

Finally, $c_i$ and $C_i$ will denote generic positive constants that assume different values at different occurrences.

\section{Nonlinear elliptic equations with {\boldmath$p$}-structure}\label{sec:pstruct}

Throughout the paper, we will consider the nonlinear elliptic equation~\cref{eq:class} with $f\in L^2(\Omega)$ and $\Omega$ being a bounded Lipschitz domain. The equation is assumed to have a $p$-structure of the form: 

\begin{assumption}\label{ass:pstruct}
The parameters $(p, r)$ and the functions $\alpha:\R^d\to\R^d, g:\R\to\R$
satisfy the properties below.
\begin{itemize}
\item Let $p\in [2,\infty)$ and $r\in(1,\infty)$. If $p<d$ then $r\leq dp/\bigl(2(d-p)\bigr)+1$.
\item The functions $\alpha$ and $g$ are continuous and satisfy the growth conditions
\begin{displaymath}
    	|\alpha(z)| \leq C_1 |z|^{p-1}\quad\text{and}\quad |g(x)| \leq C_2 |x|^{r-1},\quad\text{for all }z\in \R^d, x\in\R.
\end{displaymath}
\item The function $\alpha$ is strictly monotone with the bound
\begin{displaymath}
    	\bigl(\alpha(z)-\alpha(\tilde{z})\bigr)\cdot(z-\tilde{z})\geq c_1 |z-\tilde{z}|^p,\quad\text{for all }z,\tilde{z}\in \R^d.
\end{displaymath}
\item The function $\alpha$ is coercive with the bound
\begin{displaymath}
    	\alpha(z)\cdot z\geq c_2 |z|^p,\quad\text{for all }z\in \R^d.
\end{displaymath}
\item The function $g$ is strictly monotone and coercive with the bounds
\begin{displaymath}
    	\bigl(g(x)-g(\tilde{x})\bigr)(x-\tilde{x})\geq c_3 |x-\tilde{x}|^r\quad\text{and}\quad
	g(x)x\geq c_4 |x|^r,
	\quad\text{for all }x,\tilde{x}\in \R.
\end{displaymath}
\end{itemize}
\end{assumption}

\begin{example}\label{ex:pLaplace}
The equation~\cref{eq:resolvent} satisfies~\cref{ass:pstruct} with $\alpha(z)=|z|^{p-2}z$, $\lambda>0$, $g(x)=\lambda x$ and r=2. The same holds for equation \cref{eq:eig} with $g(x)=\lambda |x|^{p-2}x$ and $r=p$. 
\end{example}

\begin{remark}
The last assertion of \cref{ass:pstruct} is made in order to ensure that the convergence analysis of the domain decomposition is valid without employing the Poincar\'e inequality, which allows decompositions where $\partial\Omega\setminus\partial\Omega_i=\emptyset$; see~\cref{fig:sub2}. If the latter setting is excluded, then the analysis is valid for a broader class of functions $g$, especially $g=0$. 
\end{remark}

\begin{remark}
Possible generalizations of \cref{ass:pstruct}, which we omit for sake of notational simplicity, include the cases: dependence on the spatial variable and first-order terms, e.g., $\alpha(\nabla u)=\alpha(x,u,\nabla u)$; the parameter choice $p\in (2d/(d+1), 2)$, which requires an additional set of embedding results for trace spaces; other variations similar to the $p$-Laplacian, e.g.,
\begin{displaymath}
\alpha(z)=(|z_1|^{p-2}z_1,\ldots,|z_d|^{p-2}z_d),
\end{displaymath}
which only require slight reformulations of the monotonicity and coercivity conditions. 
\end{remark}

Let $V=W^{1, p}_0(\Omega)$ and define the form $a: V\times V\rightarrow \R$ by
\begin{displaymath}
    	a(u, v)=\int_\Omega \alpha(\nabla u)\cdot\nabla v+g(u)v\,\mathrm{d}x.
\end{displaymath}
 The weak form of \cref{eq:class} is to find $u\in V$ such that
 \begin{equation}\label{eq:weak}
    	a(u, v)=(f, v)_{L^2(\Omega)}, \quad\text{for all } v\in V.
 \end{equation}
The $p$-structure implies that there exist a unique weak solution of \cref{eq:weak}; see, e.g.,~\cite[Theorem 2.36]{roubicek}. A central part of the existence proof, and our convergence analysis as well, is to observe that the $p$-structure directly implies that the form $a$ is bounded, strictly monotone and coercive.
\begin{lemma}\label{lemma:aprop}
If \cref{ass:pstruct} holds, then $a: V\times V\rightarrow \R$ is well defined and satisfies the upper bound
\begin{displaymath}
	|a(u, v)| \leq  C_1 \bigl(\|\nabla u\|^{p-1}_{L^p(\Omega)^d}\|\nabla v\|^{\phantom{p-1}}_{L^p(\Omega)^d}+\|u\|_{L^r(\Omega)}^{r-1}\|v\|^{\phantom{r-1}}_{L^r(\Omega)}\bigr),
\end{displaymath}
the strict monotonicity bound
\begin{displaymath}
	a(u, u-v)-a(v, u-v)\geq c_1\bigl(\|\nabla (u-v)\|^p_{L^p(\Omega)^d}+\|u-v\|^r_{L^r(\Omega)}\bigr)
\end{displaymath}
and the coercivity bound 
\begin{displaymath}
 a(u, u)\geq c_2\bigl(\|\nabla u\|^p_{L^p(\Omega)^d}+\|u\|^r_{L^r(\Omega)}\bigr),
\end{displaymath}
for all $u,v\in V$. 
\end{lemma}

In order to conduct the convergence analysis, we also make the following additional regularity assumption on the weak solution.
\begin{assumption}\label{ass:regularity}
The weak solution $u\in V$ of~\cref{eq:weak} satisfies $\alpha(\nabla u)\in C(\overline{\Omega})^{d}$. 
\end{assumption}
Note that the above regularity assumption does \emph{not} imply that $u$ is a strong solution in $W^{2,p}(\Omega)$. A possible generalization of \Cref{ass:regularity} is discussed in \cref{rem:reg}.
\begin{example}\label{ex:reg}
Consider the equations given by the $p$-Laplacian in \cref{ex:pLaplace}. If $p\geq d$ then the weak solution $u\in V$ is also in $C(\overline{\Omega})$. If in addition $f\in L^\infty(\Omega)$ and the boundary $\partial\Omega$ is $C^{1,\beta}$, then~\cite[Theorem 1]{lieberman} yields that $u\in C^{1,\beta}(\overline{\Omega})$. The latter implies that \cref{ass:regularity} is valid in this context. 
\end{example}

Finally, we will make frequent use of the fact that, under \cref{ass:pstruct}, the standard $W^{1, p}(\Omega)$-norm is equivalent to the norm
\begin{equation}\label{eq:embedding}
u\mapsto\|\nabla u\|_{L^p(\Omega)^d}+\|u\|_{L^r(\Omega)}.
\end{equation}
For $r\geq p$ this follows directly by the Sobolev embedding theorem together with the assumed restrictions on $(p,r)$. For $r<p$ the equivalence follows by an additional bootstrap argument. 

\section{Function spaces and trace operators on {\boldmath$\Omega_i$}}\label{sec:functionspaces}

We start by considering a manifold $\mathcal{M}$ in $\R^d$, which will play the role of $\partial\Omega_i$ or $\Gamma$. The manifold $\mathcal{M}$ is said to be Lipschitz if there exist open, overlapping sets $\Theta_r$ such that
\begin{displaymath}
	        \mathcal{M}=\bigcup_{r=1}^m\Theta_r,
\end{displaymath}
where each $\Theta_r$ can be described as the graph of a Lipschitz continuous function $b_r$. More precisely, there exists 
$(d-1)$-dimensional cubes $\theta_r$ and local charts $\psi_r:\Theta_r\rightarrow\theta_r$ that are bijective and Lipschitz continuous. The charts  have the structure $\psi_r^{-1}=A^{-1}_r\circ Q_r$, where $A_r:\R^d\rightarrow\R^d$ is a coordinate transformation and 
\begin{displaymath}
Q_r:\theta_r\rightarrow \R^d:x_r\mapsto \bigl(x_r,b_r(x_r)\bigr)
\end{displaymath}
for the Lipschitz continuous map $b_r:\theta_r\rightarrow\R$. A function $\mu:\mathcal{M}\rightarrow\R$ now has the local components $\mu\circ \psi^{-1}_r$. We refer to~\cite[Section 6.2]{kufner} for further details.

On a Lipschitz manifold we may introduce a measure~\cite[Chapter 3]{Naumann2011Measure} and thus define the integral and the space $L^p(\mathcal{M})$; see, e.g.,~\cite{cohn}. From~\cite[Chapters 3.4--3.5]{cohn} it follows that $L^p(\mathcal{M})$ is a Banach space and that $L^2(\mathcal{M})$ is a Hilbert space with the inner product
\begin{displaymath}
            (\eta, \mu)_{L^2(\mathcal{M})}=\int_{\mathcal{M}}\eta\mu \,\mathrm{d}S.
\end{displaymath}
Let $\{\varphi_r\}$ be a partition of unity of $\mathcal{M}$. The integral then satisfies
\begin{equation}\label{eq:surfint}
            \int_{\mathcal{M}}\mu\,\mathrm{d}S=\sum_{r=1}^m \int_{\theta_r}(\mu\varphi_r)\circ \psi^{-1}_r|n_r|\,\mathrm{d}x,
\end{equation}
where $n_r=(\partial_1 b_r, \partial_2 b_r, \dots, \partial_{d-1}b_r, -1)$; see~\cite[Theorem 3.9]{Naumann2011Measure}. Seemingly obvious properties of the integral, including 
\begin{displaymath}
 \int_{\mathcal{M}}\mu\, \mathrm{d}S=\int_{\mathcal{M}_0}\mu\,\mathrm{d}S+\int_{\mathcal{M}\setminus\mathcal{M}_0}\mu\,\mathrm{d}S,
\end{displaymath}
relies heavily on the observation that the integral is independent of the representation $(\Theta_r,A_r,b_r)$ and the choice of partition of unity $\{\varphi_r\}$; see~\cite[Theorems 3.5 and 3.7]{Naumann2011Measure} and the comments thereafter.

The equality~\cref{eq:surfint} also shows that our integral and $L^p$-spaces are equivalent to the ones used in~\cite{kufner}. Moreover, by~\cite[Lemma 6.3.5]{kufner}, the $L^p$-norm used here is equivalent to the norm
\begin{displaymath}
	       \mu\mapsto\Big(\sum_{r=1}^m\|\mu\circ \psi^{-1}_r\|_{L^p(\theta_r)}^p\Big)^{1/p}.
\end{displaymath}
Finally, recall that for a Lipschitz manifold $\mathcal{M}$  the unit outward normal vector $\nu=(\nu^1,\ldots,\nu^d)$ is defined almost everywhere; see \cite[Section 6.10.1]{kufner}. The normal vector is given locally by $\nu\circ\psi^{-1}_r=n_r/|n_r|$ and the Lipschitz continuity of $b_r$ yields that $\nu^\ell\in L^\infty(\mathcal{M})$.
 
\begin{assumption}\label{ass:Gamma}
The boundaries $\partial\Omega_{i}$ and the interface $\Gamma$ are all $(d-1)$-dimensional Lipschitz manifolds. 
\end{assumption}

We use the notation $(\Theta_r^i, \theta_r^i,m_i, \psi_r^i, b_r^i, \phi_r^i,\nu_i)$ for the quantities related to the local representations of $\partial\Omega_i$. For later use, we note that
\begin{displaymath}
 \nu_1=-\nu_2\quad\text{ on }\Gamma.
\end{displaymath}
Next, we define the fractional Sobolev spaces on the $(d-1)$-dimensional cubes $\theta_r$. Let $0< s<1$, then $W^{s, p}(\theta_r)$ is defined as all $u\in L^p(\theta_r)$ such that
\begin{displaymath}
	        |u|_{s,\theta_r}=\Big(\int_{\theta_r}\int_{\theta_r}\frac{|u(x)-u(y)|^p}{|x-y|^{d-1+s p}}\,\mathrm{d}x\,\mathrm{d}y\Big)^{1/p}<\infty.
\end{displaymath}
The corresponding norm is given by
\begin{displaymath}
	        \|u\|_{W^{s, p}(\theta_r)}=\|u\|_{L^p(\theta_r)}+|u|_{s,\theta_r}.
\end{displaymath}
Having defined the fractional Sobolev space on $\theta^i_r$ we can also define them on $\partial\Omega_i$. For $0< s<1$, introduce
\begin{displaymath}
	        W^{s, p}(\partial\Omega_i)=\{\mu\in L^p(\partial\Omega_i): \mu\circ (\psi_r^i)^{-1}\in W^{s, p}(\theta_r^i),\text{ for } r=1,\dots,m_i\},
\end{displaymath}
equipped with the norm
\begin{displaymath}
	   \|\mu\|_{W^{s, p}(\partial\Omega_i)}=\Big(\sum_{r=1}^{m_i}\|\mu\circ (\psi_r^i)^{-1}\|^p_{W^{s, p}(\theta_r^i)}\Big)^{1/p}.
 \end{displaymath}
By the definitions of the norms, it follows directly that
\begin{displaymath}
	    \|\mu\|_{L^p(\partial\Omega_i)}\leq C\|\mu\|_{W^{s, p}(\partial\Omega_i)}.
\end{displaymath}
Furthermore, the space $W^{s, p}(\partial\Omega_i)$ is complete and reflexive; see~\cite[Definition 6.8.6]{kufner} and the comment thereafter. 
Next, we recapitulate the trace theorem for $W^{1,p}$-functions on Lipschitz domains; see, e.g.,~\cite[Theorems 6.8.13 and 6.9.2]{kufner}.
\begin{lemma}\label{lemma:trace}
If the \cref{ass:pstruct,ass:Gamma} are valid, then there exists a surjective bounded linear operator $T_{\partial\Omega_i}:W^{1 ,p}(\Omega_i)\rightarrow W^{1-1/p, p}(\partial\Omega_i)$ such that $T_{\partial\Omega_i}u=u|_{\partial\Omega_i}$ when $u\in C^\infty(\overline{\Omega}_i)$. The operator $T_{\partial\Omega_i}$ has a bounded linear right inverse $R_{\partial\Omega_i}:W^{1-1/p, p}(\partial\Omega_i)\rightarrow W^{1 ,p}(\Omega_i)$.
\end{lemma}
We can then define the Sobolev spaces on $\Omega_i$ required for the domain decomposition, namely 
\begin{displaymath}
V_i^0=W^{1, p}_0(\Omega_i)
\quad\text{and}\quad
V_i=\{v\in W^{1, p}(\Omega_i): (T_{\partial\Omega_i} v)|_{\partial\Omega_i\setminus\Gamma}=0\}.
\end{displaymath}
The spaces are equipped with the norm
\begin{displaymath}
    \|v\|_{V_i}=\|\nabla v\|_{L^p(\Omega_i)^d}+\|v\|_{L^r(\Omega_i)}.
\end{displaymath}
As for~\cref{eq:embedding}, this norm is equivalent to the standard $W^{1, p}(\Omega_i)$-norm under~\cref{ass:pstruct}. Furthermore, the spaces $V_i^0$ and $V_i$ are reflexive Banach spaces.

\section{Function spaces and trace operators on {\boldmath$\Gamma$}}\label{sec:LM}

The $L^p$-form of the Lions--Magenes spaces can be defined as
\begin{displaymath}
        \Lambda_i=\{\mu\in L^p(\Gamma): E_i\mu\in W^{1-1/p, p}(\partial\Omega_i)\}, 
        \quad\text{with}\quad
         \|\mu\|_{\Lambda_i}=\|E_i\mu\|_{W^{1-1/p, p}(\partial\Omega_i)}.
 \end{displaymath}
Here, $E_i\mu$ denotes the extension by zero of $\mu$ to $\partial\Omega_i$. We also define the trace space 
\begin{displaymath}
	\Lambda=\{\mu\in L^p(\Gamma): \mu\in\Lambda_i,\text{ for } i=1,2\},
	\quad\text{with}\quad
	\|\mu\|_{\Lambda}=\|\mu\|_{\Lambda_1}+\|\mu\|_{\Lambda_2}.
\end{displaymath}

\begin{lemma}\label{lem:LionsMagenes}
If the \cref{ass:pstruct,ass:Gamma} hold, then $\Lambda_i$ and $\Lambda$ are reflexive Banach spaces. 
\end{lemma}
\begin{proof}
Observe that $E_i$ is a linear isometry from $\Lambda_i$ onto 
\begin{equation}\label{eq:ranEi}
        \ran(E_i)=\{\mu\in {W^{1-1/p, p}(\partial\Omega_i)}: \mu|_{\partial\Omega\setminus\Gamma}=0\}.
\end{equation}
Next, consider a sequence $\{\mu^k\}\subset  \ran(E_i)$ such that $\mu^k\to \mu$ in $W^{1-1/p, p}(\partial\Omega_i)$. Then $\mu^k|_{\Omega_i\setminus\Gamma}=0$ and $\mu^k\to \mu$ in $L^p(\partial\Omega_i)$, which implies that 
\begin{equation}\label{eq:Ei}
	\int_{\partial\Omega_i\setminus\Gamma}|\mu|^p\,\mathrm{d}S
		=\int_{\partial\Omega_i\setminus\Gamma}|\mu^k-\mu|^p\,\mathrm{d}S
		\leq\int_{\partial\Omega_i}|\mu^k-\mu|^p\,\mathrm{d}S\to 0,\quad\text{as }k\to\infty.
\end{equation}
Hence, $\mu\in\ran(E_i)$ and consequently $\ran(E_i)$ is a closed subset of $W^{1-1/p, p}(\partial\Omega_i)$. The space $\Lambda_i$ is therefore isomorphic to a closed subset of the reflexive Banach space $W^{1-1/p, p}(\partial\Omega_i)$, i.e., $\Lambda_i$ is complete and reflexive~\cite[Chapter 8, Theorem 15]{lax}.

To prove that the same holds true for $\Lambda$ introduce the reflexive Banach space $X=W^{1-1/p, p}(\partial\Omega_1)\times W^{1-1/p, p}(\partial\Omega_2)$, with the norm
\begin{displaymath}
        \|(\mu_1, \mu_2)\|_X=\|\mu_1\|_{W^{1-1/p, p}(\partial\Omega_1)}+\|\mu_2\|_{W^{1-1/p, p}(\partial\Omega_2)},
\end{displaymath}
and the operator $E:\Lambda\rightarrow X$ defined by $E\mu=(E_1\mu, E_2\mu)$. As $E$ is a linear isometry from $\Lambda$ onto
\begin{displaymath}
 \ran(E)=\{(\mu_1, \mu_2)\in X: {\mu_1}|_{\partial\Omega_1\setminus\Gamma}=0,\, {\mu_2}|_{\partial\Omega_2\setminus\Gamma}=0,\, {\mu_1}|_{\Gamma}={\mu_2}|_{\Gamma} \},
\end{displaymath}
it is again sufficient to prove that $\ran(E)$ is a closed subset of $X$. Let $\{(\mu_1^k,\mu_2^k)\}\subset\ran(E)$ be a convergent sequence in $X$ with the limit $(\mu_1,\mu_2)$. By the same argument as~\cref{eq:Ei}, we obtain that $\mu_i|_{\Omega_i\setminus\Gamma}=0$. As ${\mu^k_1}|_{\Gamma}={\mu^k_2}|_{\Gamma}$, we also have the limit
 \begin{displaymath}
        \int_\Gamma|\mu_1-\mu_2|^p\,\mathrm{d}S
        \leq 2^{p-1}\big(\int_\Gamma|\mu_1^k-\mu_1|^p\,\mathrm{d}S
         +\int_\Gamma|\mu_2^k-\mu_2|^p\,\mathrm{d}S\big)\to 0,\quad\text{as }k\to\infty,
\end{displaymath}
i.e., ${\mu_1}|_{\Gamma}={\mu_2}|_{\Gamma}$ in $L^p(\Gamma)$ and we obtain that $(\mu_1,\mu_2)\in\ran(E)$. Thus, $\ran(E)$ is closed and $\Lambda$ is therefore a reflexive Banach space.  
\end{proof}
\begin{lemma}
If the \cref{ass:pstruct,ass:Gamma} hold, then $\Lambda_i$ and $\Lambda$ are dense in $L^2(\Gamma)$.
\end{lemma}
The proof of the lemma is almost identical to the proof of \cite[Theorem 6.6.3]{kufner} and is therefore left out. 
\begin{remark}
We conjecture that $\Lambda_1=\Lambda_2$. However, we will move on to a $L^2(\Gamma)$-framework for which it is not necessary to make this identification. 
\end{remark}
Collecting these results yield the Gelfand triplets 
\begin{displaymath}
	\Lambda_i \overset{d}{\hookrightarrow} L^2(\Gamma) \cong L^2(\Gamma)^* \overset{d}{\hookrightarrow}  \Lambda_i^*
	\quad \text{and}\quad 
	\Lambda \overset{d}{\hookrightarrow} L^2(\Gamma) \cong L^2(\Gamma)^* \overset{d}{\hookrightarrow}  \Lambda^*.
\end{displaymath}
For future use, we introduce the Riesz isomorphism on  $L^2(\Gamma)$ given by
\begin{displaymath}
	J: L^2(\Gamma)\to L^2(\Gamma)^*: \mu\mapsto (\mu,\cdot)_{ L^2(\Gamma)},
\end{displaymath}
which satisfies the relations
\begin{displaymath}
	\langle J\eta,\mu_i\rangle_{\Lambda_i^*\times \Lambda_i}= (\eta,\mu_i)_{L^2(\Gamma)}
	\quad\text{and}\quad
	\langle J\eta,\mu\rangle_{\Lambda^*\times \Lambda}= (\eta,\mu)_{ L^2(\Gamma)},
\end{displaymath}
for all $\eta\in L^2(\Gamma)$, $\mu_i\in \Lambda_i$ and $\mu\in \Lambda$. Here, $\langle\cdot,\cdot\rangle_{X^*\times X}$
denotes the dual pairing between a Banach space $X$ and its dual $X^*$. In the following we will drop the subscripts on the dual parings. 
  
In order to relate the spaces $V_i$ and $\Lambda_i$, we observe that for $v\in V_i$ one has $T_{\partial\Omega_i}v\in\ran(E_i)$; see~\cref{eq:ranEi}. Hence, the trace operator
\begin{displaymath}
T_i:V_i\to\Lambda_i:v\mapsto \bigl(T_{\partial\Omega_i}v\bigr)|_{\Gamma}
\end{displaymath}
is well defined. We also introduce the linear operator
\begin{displaymath}
R_i:\Lambda_i\to V_i:\mu\mapsto R_{\partial\Omega_i} E_i\mu.
\end{displaymath}

\begin{lemma}
If the \cref{ass:pstruct,ass:Gamma} hold, then $T_i$ and $R_i$ are bounded, and $R_i$ is a right inverse of $T_i$.
\end{lemma}

\begin{proof}
For $v\in V_i$ and $\mu\in\Lambda_i$ we have, by \cref{lemma:trace}, that 
\begin{gather*}
   	\|T_iv\|_{\Lambda_i} =\|E_i\bigl((T_{\partial\Omega_i}v)|_{\Gamma}\bigr)\|_{W^{1-1/p,p}(\partial\Omega_i)}
		                          =\|T_{\partial\Omega_i}v\|_{W^{1-1/p,p}(\partial\Omega_i)}\leq C_i\|v\|_{V_i}\\[5pt]
	 \text{and}\quad
	 \|R_i\mu\|_{V_i} =  \|R_{\partial\Omega_i} E_i\mu\|_{V_i}\leq C_i \|E_i\mu\|_{W^{1-1/p,p}(\partial\Omega_i)}=C_i\|\mu\|_{\Lambda_i}.
\end{gather*}
Hence, the linear operators $T_i$ and $R_i$ are bounded. Furthermore, for every $\mu\in\Lambda_i$ we have 
\begin{displaymath}
T_iR_i\mu=\bigl(T_{\partial\Omega_i}R_{\partial\Omega_i} E_i\mu\bigr)|_{\Gamma} = (E_i\mu)|_{\Gamma}=\mu,
\end{displaymath}
i.e., $R_i$ is a right inverse of $T_i$.
\end{proof}
We continue by deriving a few useful properties related to the operator $T_i$. 
\begin{lemma}\label{lemma:L1}
If the \cref{ass:pstruct,ass:Gamma} hold and $v\in V$, then $\mu = T_1v|_{\Omega_1}=T_2v|_{\Omega_2}$ is an element in $\Lambda$.
\end{lemma}

\begin{proof}
Let $v\in V$. As $C^\infty_0(\Omega)$ is dense in $V$, there exists a sequence $\{v^k\}\subset C^\infty_0(\Omega)$ such that $v^k\to v$ in $V$. Set $v_i=v|_{\Omega_i}$ and  $v_i^k=v^k|_{\Omega_i}$. Clearly, $T_1v^k_1=T_2v^k_2$. Since $v^k\to v$ in $V$, we also have that $v_i^k\to v_i$ in $V_i$. The continuity of $T_i$ then implies that $T_iv_i^k\to T_iv_i$ in $\Lambda_i$. Putting this together gives us 
\begin{displaymath}
\Lambda_1\ni T_1v_1=\lim_{k\to\infty}T_1v_1^k= \lim_{k\to\infty}T_2v_2^k=T_2v_2\in\Lambda_2\quad\text{in }L^p(\Gamma).
\end{displaymath}
If we now define $\mu=T_1v_1=T_2v_2$, then $\mu$ is an element in $\Lambda=\Lambda_1\cap\Lambda_2$.
\end{proof}

\begin{lemma}\label{lemma:L2}
Let the \cref{ass:pstruct,ass:Gamma} hold. If two elements $v_{1}\in V_1$ and $v_{2}\in V_2$ satisfies $T_1v_1=T_2v_2$, then
$v=\{v_1 \textrm{ on } \Omega_1;\, v_2 \textrm{ on } \Omega_2\}$ is an element in $V$.
\end{lemma}

\begin{proof}
It is clear that $v\in L^p(\Omega)$. For each component $1\leq \ell\leq d$, there exists a weak derivative $\partial_\ell v_i\in L^p(\Omega_i)$ of $v_i\in V_i\subset W^{1,p}(\Omega_i)$. If we define
\begin{displaymath}
       z_\ell=\{\partial_\ell v_1 \textrm{ on } \Omega_1;\, \partial_\ell v_2 \textrm{ on } \Omega_2\},
\end{displaymath}
then $z_\ell\in L^p(\Omega)$. Let $w\in C^\infty_0(\Omega)$ and set $w_i=w|_{\Omega_i}\in C^\infty(\Omega_i)$. The $W^{1,p}(\Omega_i)$-version of Green's formula~\cite[Section 3.1.2]{necas} yields that
 \begin{align*}
        \int_{\Omega} z_\ell w\,\mathrm{d}x&=\sum_{i=1}^2\int_{\Omega_i} \partial_\ell v_i\, w_i \,\mathrm{d}x
        	          =\sum_{i=1}^2-\int_{\Omega_i} v_i\,\partial_\ell w_i\, \mathrm{d}x+ 
	          \int_{\partial\Omega_i} (T_{\partial\Omega_i}v_i)w_i\nu_i^\ell\,\mathrm{d}S \\
        &=-\int_{\Omega} v \partial_\ell w \,\mathrm{d}x+\sum_{i=1}^2\int_{\Gamma} (T_iv_i)w\nu_i^\ell\, \mathrm{d}S
           =-\int_{\Omega} v\partial_\ell w\,\mathrm{d}x,
\end{align*}
i.e., $z_\ell$ is the $\ell$th weak partial derivative of $v$. By construction $T_{\partial\Omega}v=0$, and $v$ is therefore an element in $V$.
\end{proof}  
  
\section{Transmission problem and the Robin--Robin method}\label{sec:weakform}

In order to state the Robin--Robin method we reformulate the nonlinear elliptic equation~\cref{eq:weak} on $\Omega$ into two equations on $\Omega_i$ connected via $\Gamma$, i.e., we consider a nonlinear transmission problem. To this end, on each $V_i$ we define $a_i: V_i\times V_i\rightarrow \R$ by 
\begin{displaymath}
    	a_i(u_i, v_i)=\int_{\Omega_i} \alpha(\nabla u_i)\cdot\nabla v_i+g(u_i)v_i\,\mathrm{d}x.
\end{displaymath}
We also define $f_i=f|_{\Omega_i}\in L^2(\Omega_i)$.
\begin{lemma}\label{lemma:aiprop}
If the \cref{ass:pstruct,ass:Gamma} hold, then $a_i: V_i\times V_i\rightarrow \R$  is well defined and satisfies the growth, strict monotonicity and coercivity bounds stated in \cref{lemma:aprop}, with the terms $(a,V,\Omega)$ replaced by $(a_i,V_i,\Omega_i)$.
\end{lemma}
The weak form of the nonlinear transmission problem is then to find $(u_1,u_2)\in V_1\times V_2$ such that 
\begin{equation}\label{eq:tran}
\left\{\begin{aligned}
	     a_i(u_i, v_i)&=(f_i, v_i)_{L^2(\Omega_i)}, & & \text{for all } v_i\in V_i^0, i=1,2,\\
	     T_1u_1&=T_2u_2, & &\\
	     \textstyle\sum_{i=1}^2 a_i(& u_i, R_i\mu)-(f_i, R_i\mu)_{L^2(\Omega_i)}=0,& &\text{for all }\mu\in \Lambda. 
\end{aligned}\right.
\end{equation}
The framework given in \cref{sec:LM} enables us to prove equivalence between the nonlinear elliptic equation and the nonlinear transmission problem along the same lines as done for linear equations~\cite[Lemma~1.2.1]{quarteroni}.

\begin{theorem}\label{thm:TPequivWE}
Let \cref{ass:pstruct,ass:Gamma} hold. If $u\in V$ solves \cref{eq:weak}, then $(u_1,u_2)=(u|_{\Omega_1},u|_{\Omega_2})$ solves \cref{eq:tran}. Conversely, if $(u_1, u_2)$ solves \cref{eq:tran}, then $u=\{u_1 \textrm{ on } \Omega_1;\, u_2 \textrm{ on } \Omega_2\}$ solves \cref{eq:weak}.
\end{theorem}

\begin{proof}
Assume that $u\in V$ solves~\cref{eq:weak} and define $(u_1,u_2)=(u|_{\Omega_1},u|_{\Omega_2})\in V_1\times V_2$. For $v_i\in V_i^0$ we can extend by zero to $w_i\in V$, by using \cref{lemma:L2}. Then
\begin{displaymath}
	    a_i(u_i, v_i)=a(u, w_i)=(f, w_i)_{L^2(\Omega)}=(f_i, v_i)_{L^2(\Omega_i)}. 
\end{displaymath}
Moreover, $T_1u_1=T_2u_2$ follows immediately from \cref{lemma:L1}. For an arbitrary $\mu\in \Lambda$ let $v_i=R_i\mu$. Then, by \cref{lemma:L2}, $v=\{v_1 \textrm{ on } \Omega_1;\, v_2 \textrm{ on } \Omega_2\}$ is an element in $V$ and 
\begin{displaymath}
	a_1(u_1,R_1\mu)+a_2(u_2,R_2\mu)=a(u,v)=(f,v)_{L^2(\Omega)}
	    =(f_1, R_1\mu)_{L^2(\Omega_1)}+(f_2, R_2\mu)_{L^2(\Omega_2)}.
\end{displaymath}
This proves that $(u_1,u_2)$ solves~\cref{eq:tran}. Conversely, let $(u_1, u_2)\in V_1\times V_2$ be a solution to~\cref{eq:tran}  and define $u=\{u_1 \textrm{ on } \Omega_1;\, u_2 \textrm{ on } \Omega_2\}$. By \cref{lemma:L2}, we have that $u\in V$. Next, consider $v\in V$ and let $v_i=v|_{\Omega_i}\in V_i$. From \cref{lemma:L1} we have that  
$\mu=T_iv_i$ is well defined and $\mu\in\Lambda$. The observation that $v_i-R_i\mu\in V_i^0$, for $i=1, 2$, then implies the equality
\begin{align*}
		a(u, v) &=\sum_{i=1}^2 a_i(u_i, v_i-R_i\mu)+a_i(u_i, R_i\mu)=\sum_{i=1}^2(f_i, v_i-R_i\mu)_{L^2(\Omega_i)}+a_i(u_i, R_i\mu) \\
		&=\sum_{i=1}^2(f_i, v_i)_{L^2(\Omega_i)}+a_i(u_i, R_i\mu)-(f_i, R_i\mu)_{L^2(\Omega_i)}=(f, v)_{L^2(\Omega)}.
\end{align*}
As $v$ can be chosen arbitrarily, $u$ solves~\cref{eq:weak}. 
\end{proof}

\begin{remark}
As the nonlinear elliptic equation~\cref{eq:weak} has a unique weak solution, \cref{thm:TPequivWE} implies that the same holds true for the nonlinear transmission problem~\cref{eq:tran}.
\end{remark}

\begin{remark}
The equivalence can easily be generalized to more than two subdomains without cross points as in \cref{fig:sub3}. However, for domain decompositions with cross points such as in \cref{fig:sub4}, the situation is more delicate. A proof of the equivalence for quasilinear equations in $H^1(\Omega)$ can be found in~\cite{schreiber}. This result can most likely be generalized to our $W^{1,p}(\Omega)$-setting, but we will study this aspects elsewhere. 
\end{remark}

In order to approximate the weak solution $(u^{n}_1,u^{n}_2)\in V_1\times V_2$ of the nonlinear transmission problem in a parallell fashion, we consider the Robin--Robin method. Multiplying by test functions and formally applying Green's formula to the equations~\cref{eq:RobinStrong} yield that the weak form of the method is given by finding  $(u^{n}_1,u^{n}_2)\in V_1\times V_2$, for $n=1,2,\ldots,$ such that
\begin{equation}\label{eq:Robin}
\left\{\begin{aligned}
	&a_1(u_1^{n+1}, v_1)=(f_1, v_1)_{L^2(\Omega_1)}, & &\text{for all } v_1\in V_1^0\\
        &a_1(u_1^{n+1}, R_1\mu)-(f_1, R_1\mu)_{L^2(\Omega_1)}+a_2(u_2^n, R_2\mu)  & & \\
			&\qquad-(f_2, R_2\mu)_{L^2(\Omega_2)}=s(T_{2}u_2^n-T_{1}u_1^{n+1}, \mu)_{L^2(\Gamma)}, & &\text{for all }\mu\in \Lambda,\\[5pt]
	&a_2(u_2^{n+1}, v_2)=(f_2, v_2)_{L^2(\Omega_2)}, & &\text{for all } v_2\in V_2^0\\
	&a_2(u_2^{n+1}, R_2\mu)-(f_2, R_2\mu)_{L^2(\Omega_2)}+a_1(u_1^{n+1}, R_1\mu) & &\\
			&\qquad-(f_1, R_1\mu)_{L^2(\Omega_1)}=s(T_{1}u_1^{n+1}-T_{2}u_2^{n+1}, \mu)_{L^2(\Gamma)}, & &\text{for all }\mu\in \Lambda,
\end{aligned}\right.
\end{equation}
where $u^0_2\in V_2$ is an initial guess and $s>0$ is the given method parameter. 

\section{Interface formulations}\label{sec:interface}

The ambition is now to reformulate the nonlinear transmission problem and the Robin--Robin method, which are all given on the domains $\Omega_{i}$, into problems and methods only stated on the interface $\Gamma$. As a preparation, we observe that nonlinear elliptic equations on $\Omega_{i}$ with inhomogeneous Dirichlet conditions have unique weak solutions.
\begin{lemma}\label{lemma:Fi}
If the \cref{ass:pstruct,ass:Gamma} hold, then for each $\eta\in \Lambda_i$ there exists a unique $u_i\in W^{1,p}(\Omega_i)$ such that
\begin{equation}\label{eq:Fi}
        a_i(u_i, v_i)=(f_i, v_i)_{L^2(\Omega_i)},\quad\text{for all }v_i\in V^0_i,
\end{equation}
and $T_{\partial\Omega_{i}}u_i=E_i\eta$ in $W^{1-1/p,p}(\partial\Omega_i)$.
\end{lemma}
The proof can, e.g., be found in \cite[Theorem 2.36]{roubicek}. With the notation of \cref{lemma:Fi}, consider the operator 
\begin{displaymath}
F_i:\eta\mapsto u_i,
\end{displaymath}
i.e., the map from a given boundary value on $\Gamma$ to the corresponding weak solution of the nonlinear elliptic problem~\cref{eq:Fi} on $\Omega_i$.
From the statement of \cref{lemma:Fi} we see that 
\begin{displaymath}
F_i:\Lambda_i\to V_i\quad\text{and}\quad T_iF_i\eta=\eta\text{ for }\eta\in\Lambda_i.
\end{displaymath}
In other words, the operator $F_i$ is a nonlinear right inverse to $T_i$. This property will be frequently used as it, together with the boundedness and linearity of $T_i$, gives rise to bounds of the forms 
\begin{displaymath}
\|\eta\|_{\Lambda_i}\leq C_i \|F_i\eta\|_{V_i}\quad\text{and}\quad\|\eta-\mu\|_{\Lambda_i}\leq C_i \|F_i\eta-F_i\mu\|_{V_i}.
\end{displaymath}
We can now define the nonlinear Steklov--Poincar\'e operators as
\begin{gather*}
    \langle S_i\eta, \mu\rangle= a_i(F_i\eta, R_i\mu)-(f_i, R_i\mu)_{L^2(\Omega_i)},\quad\text{for all }\eta,\mu\in\Lambda_i,\quad\text{and}\\[5pt]
    \langle S\eta, \mu\rangle=\sum_{i=1}^2\langle S_i\eta, \mu\rangle=\sum_{i=1}^2  a_i(F_i\eta, R_i\mu)-(f_i, R_i\mu)_{L^2(\Omega_i)},\quad\text{for all }\eta,\mu\in\Lambda. 
\end{gather*}
Thus, we may restate the nonlinear transmission problem~\cref{eq:tran} as the Steklov--Poincar\'e equation, i.e., finding $\eta\in\Lambda$ such that
\begin{equation}\label{eq:speq}
    \langle S\eta, \mu\rangle=0,\quad\text{for all }\mu\in\Lambda.
\end{equation}
That the reformulation is possible follows directly from the definitions of the operators $F_i$ and $S$, but we state this as a lemma for future reference.
\begin{lemma}
Let the \cref{ass:pstruct,ass:Gamma} hold. If $(u_1,u_2)$ solves~\cref{eq:tran}, then $\eta=T_1u_1=T_2 u_2$ solves~\cref{eq:speq}. Conversely, if $\eta$ solves~\cref{eq:speq}, then $(u_1,u_2)=(F_1\eta,F_2\eta)$ solves~\cref{eq:tran}.
\end{lemma}

We next turn to the Robin--Robin method, which is equivalent to the Peaceman--Rachford splitting on the interface $\Gamma$. The weak form of the  splitting is given by finding $(\eta^{n}_{1}, \eta^{n}_{2})\in \Lambda_{1}\times\Lambda_{2}$, for $n=1,2,\ldots,$ such that 
\begin{equation}\label{eq:PR}
\left\{\begin{aligned}
 	\langle(sJ+S_1)\eta_1^{n+1}, \mu\rangle&=\langle (sJ-S_2)\eta_2^n, \mu\rangle, \\
	 \langle(sJ+S_2)\eta_2^{n+1}, \mu\rangle&=\langle (sJ-S_1)\eta_1^{n+1}, \mu\rangle,
\end{aligned}\right.
\end{equation}
for all $\mu\in \Lambda$, where $\eta_2^0\in \Lambda_2$ is an initial guess. 

The observation that the Robin--Robin method and the Peaceman--Rachford splitting are equivalent for linear elliptic equations was made in~\cite{Agoshkov83}. The equivalence was also utilized in~\cite[Section 4.4.1]{discacciati} for the linear setting of the Stokes--Darcy coupling.

\begin{lemma}\label{lem:RobinPRequiv}
Let the \cref{ass:pstruct,ass:Gamma} be valid. If $(u^{n}_1, u^{n}_2)_{n\geq 1}$ is a weak Robin--Robin approximation~\cref{eq:Robin}, then $(\eta^{n}_{1}, \eta^{n}_{2})_{n\geq 1}=(T_1u^{n}_1, T_2 u^{n}_2)_{n\geq 1}$  is a weak Peaceman--Rachford approximation~\cref{eq:PR}, with $\eta^0_2=T_2u^0_2$. Conversely, if $(\eta^{n}_{1}, \eta^{n}_2)_{n\geq 1}$ is given by~\cref{eq:PR}, then $(u^{n}_1, u^{n}_2)_{n\geq 1}=(F_1\eta^{n}_{1}, F_2\eta^{n}_{2})_{n\geq 1}$ fulfils~\cref{eq:Robin}, with $u^0_2=F_2\eta^0_2$.
\end{lemma}

\begin{proof}
First assume that $(u^{n}_1, u^{n}_2)_{n\geq 1}\subset V_1\times V_2$ is a weak Robin--Robin approximation and define $\eta^n_i=T_iu^n_i\in \Lambda_i$. This definition, the existence of a unique solution of~\cref{eq:Fi} and the first and third assertions of~\cref{eq:Robin} yield the identification $u_i^n=F_i\eta_i^n$. Inserting this into the second and fourth assertion of~\cref{eq:Robin} gives us 
\begin{align*}
	   &s(T_1F_1\eta_1^{n+1}, \mu)_{L^2(\Gamma)}+a_1(F_1\eta_1^{n+1}, R_1\mu)-(f, R_1\mu)_{L^2(\Omega_1)}\\ 
	   &\quad=s(T_2F_2\eta_2^n, \mu)_{L^2(\Gamma)}-a_2(F_2\eta_2^n,R_2\mu)+(f, R_2\mu)_{L^2(\Omega_2)},\quad\text{and}\\[5pt]
	   &s(T_2F_2\eta_2^{n+1}, \mu)_{L^2(\Gamma)}+a_2(F_2\eta_2^{n+1}, R_2\mu)-(f, R_2\mu)_{L^2(\Omega_2)}\\ 
	   &\quad=s(T_1F_1\eta_1^{n+1}, \mu)_{L^2(\Gamma)}-a_1(F_1\eta_1^{n+1},R_1\mu)+(f, R_1\mu)_{L^2(\Omega_1)},
\end{align*}
for all $\mu\in\Lambda$, which is precisely the weak form of the Peaceman--Rachford splitting~\cref{eq:PR}, with $\eta^0_2=T_2u^0_2$. Conversely, suppose that $(\eta^{n}_{1}, \eta^{n}_2)_{n\geq 1}\subset  \Lambda_1\times\Lambda_2$ is a weak Peaceman--Rachford approximation and define $u_i^n=F_i\eta^n_i\in V_i$. Inserting this into~\cref{eq:PR} directly gives that $(u^{n}_1, u^{n}_2)_{n\geq 1}$, with $u_2^0=F_2\eta^0_2$, is a weak Robin--Robin approximation~\cref{eq:Robin}.
\end{proof}

\begin{remark}
For now we do not know if the weak Robin--Robin and Peaceman--Rachford approximations actually exist, but we will return to this issue shortly.
\end{remark}

\section{Properties of the nonlinear Steklov--Poincar\'e operators}\label{sec:SP}

We proceed by deriving the central properties of the Steklov--Poincar\'e operators $S_i, S$ when interpreted as maps from $\Lambda_i, \Lambda$ into the corresponding dual spaces.

\begin{lemma}
If the \cref{ass:pstruct,ass:Gamma} hold, then $S_i: \Lambda_i\rightarrow \Lambda_i^*$ and $S:\Lambda\rightarrow\Lambda^*$ are well defined. 
\end{lemma}
\begin{proof}
The linearity of the functionals $S_i\eta$ and $S\eta$ follow by definition. As $F_i\eta\in V_i$ we have, by \cref{lemma:aiprop}, that 
\begin{align*}
	|\langle S_i\eta, \mu \rangle|&\leq |a_i(F_i\eta, R_i\mu)|+|(f_i, R_i\mu)_{L^2(\Omega_i)}|\\
	&\leq c_i(\|\nabla F_i\eta\|^{p-1}_{L^p(\Omega_i)^d}\|\nabla R_i\mu\|^{\phantom{p-1}}_{L^p(\Omega_i)^d}+\|F_i\eta\|_{L^r(\Omega_i)}^{r-1}\|R_i\mu\|^{\phantom{r-1}}_{L^r(\Omega_i)})\\
	&\quad+\|f_i\|_{L^2(\Omega_i)}\|R_i\mu\|_{L^2(\Omega_i)}\\
	&\leq C_i(\|\nabla F_i\eta\|_{V_i},\|f_i\|_{L^2(\Omega_i)})\|R_i\mu\|_{V_i}\leq C_i\|\mu\|_{\Lambda_i},
\end{align*}
for all $\mu\in\Lambda_i$. Thus, $S_i\eta$ is a bounded functional on $\Lambda_i$. The boundedness of $S\eta$ follows directly by summing up the bounds for $S_i$ . 
\end{proof}

\begin{lemma}\label{lemma:monotonicity}
If the \cref{ass:pstruct,ass:Gamma} hold, then the operators $S_i: \Lambda_i\rightarrow \Lambda_i^*$ and $S:\Lambda\rightarrow\Lambda^*$ are strictly monotone with 
\begin{equation*}
      \langle S_i\eta-S_i\mu, \eta-\mu \rangle\geq c_i\bigl(\|\nabla(F_i \eta-F_i \mu)\|_{L^p(\Omega_i)^d}^p+\|F_i \eta-F_i \mu\|_{L^r(\Omega_i)}^r\bigr), 
\end{equation*}
for all $\eta, \mu\in\Lambda_i$, and
\begin{equation*}
         \langle S\eta-S\mu, \eta-\mu \rangle\geq c\sum_{i=1}^2 \bigl(\|\nabla(F_i \eta-F_i \mu)\|_{L^p(\Omega_i)^d}^p+\|F_i \eta-F_i \mu\|_{L^r(\Omega_i)}^r\bigr), 
\end{equation*}
for all $\eta, \mu\in\Lambda$, respectively. 
\end{lemma}

\begin{proof}
Since, $w_i=R_i(\eta-\mu)-(F_i\eta-F_i\mu)\in V_i^0$, for all $\eta,\mu\in\Lambda_i$, we have according to the definition of $F_i$ that 
\begin{equation}\label{eq:trick}
a_i(F_i\eta,w_i)-a_i(F_i\mu,w_i)=0.
\end{equation}
By this equality and \cref{lemma:aiprop}, it follows that 
\begin{align*}
	        \langle S_i\eta-S_i\mu &, \eta-\mu \rangle =a_i\bigl(F_i\eta, R_i(\eta-\mu)\bigr)-a_i\bigl(F_i\mu, R_i(\eta-\mu)\bigr) \\
	        &=a_i(F_i\eta,w_i)+a_i(F_i\eta, F_i\eta-F_i\mu)- a_i(F_i\mu,w_i)-a_i(F_i\mu, F_i\eta-F_i\mu)\\
	        &\geq c_i\bigl(\|\nabla(F_i\eta-F_i\mu)\|_{L^p(\Omega_i)^d}^p+\|F_i\eta-F_i\mu\|_{L^r(\Omega_i)}^r\bigr),
\end{align*}
for all $\eta, \mu\in\Lambda_i$, which proves the monotonicity bound for $S_i$. The bound for $S$ follows directly by summing up the bounds for $S_i$.
\end{proof}

\begin{lemma}\label{lemma:coercivity}
If the \cref{ass:pstruct,ass:Gamma} hold, then the operators $S_i: \Lambda_i\rightarrow \Lambda_i^*$ and $S:\Lambda\rightarrow\Lambda^*$ are coercive, i.e, 
\begin{displaymath}
\lim_{\|\eta\|_{\Lambda_i}\to\infty}\frac{\langle S_i\eta ,\eta\rangle}{\|\eta\|_{\Lambda_i}}=\infty
\quad\text{and}\quad
\lim_{\|\eta\|_{\Lambda}\to\infty}\frac{\langle S\eta ,\eta\rangle}{\|\eta\|_{\Lambda}}=\infty.
\end{displaymath}
\end{lemma}

\begin{proof}
As $\| F_i\eta\|_{V_i}=\|\nabla F_i \eta\|_{L^p(\Omega_i)^d}+\|F_i \eta\|_{L^r(\Omega_i)}\geq c_i\|\eta\|_{\Lambda_i}$, we have that 
\begin{displaymath}
P(\|\nabla F_i \eta\|_{L^p(\Omega_i)^d}, \|F_i \eta\|_{L^r(\Omega_i)})\to\infty,\quad\text{as }\|\eta\|_{\Lambda_i}\to\infty,
\end{displaymath}
where $P(x,y)=(x^p+y^r)/(x+y)$. In particular, we assume from now on that 
\begin{displaymath}
P(\|\nabla F_i \eta\|_{L^p(\Omega_i)^d}, \|F_i \eta\|_{L^r(\Omega_i)})\geq \|f_i\|_{L^2(\Omega_i)}.
\end{displaymath}
By observing that $R_i\eta-F_i\eta\in V_i^0$, \cref{lemma:aiprop} yields the lower bound
\begin{align*}
	\langle S_i\eta ,\eta\rangle&=a_i(F_i\eta, F_i\eta)+a_i(F_i\eta, R_i\eta-F_i\eta)-(f_i, R_i\eta)_{L^2(\Omega_i)}\\
		&=a_i(F_i\eta, F_i\eta)+(f_i, R_i\eta-F_i\eta)_{L^2(\Omega_i)}-(f_i, R_i\eta)_{L^2(\Omega_i)}\\
	        &\geq c_i(\|\nabla F_i\eta\|^p_{L^p(\Omega_i)^d}+\|F_i\eta\|^r_{L^r(\Omega_i)})-(f_i, F_i\eta)_{L^2(\Omega_i)}\\
	        &\geq c_i P(\|\nabla F_i \eta\|_{L^p(\Omega_i)^d}, \|F_i \eta\|_{L^r(\Omega_i)})\|F_i\eta\|_{V_i}-\|f_i\|_{L^2(\Omega_i)}\|F_i\eta\|_{V_i}\\
	        &\geq c_i\bigl( P(\|\nabla F_i \eta\|_{L^p(\Omega_i)^d}, \|F_i \eta\|_{L^r(\Omega_i)})-\|f_i\|_{L^2(\Omega_i)}\bigr)\|\eta\|_{\Lambda_i},
\end{align*}
which implies that $S_i$ is coercive. For $S$ we obtain that
\begin{displaymath}
	 \frac{\langle S\eta ,\eta\rangle}{\|\eta\|_\Lambda}\geq 
	 \frac{\sum_{i=1}^2\bigl(c_i\, P(\|\nabla F_i \eta\|_{L^p(\Omega_i)^d}, \|F_i \eta\|_{L^r(\Omega_i)})-\|f_i\|_{L^2}\bigr)\|\eta\|_{\Lambda_i}}{\|\eta\|_{\Lambda_1}+\|\eta\|_{\Lambda_2}},
\end{displaymath}
which tends to infinity as $\|\eta\|_\Lambda$ tends to infinity. Thus, $S$ is also coercive. 
\end{proof}

In order to prove that the operators $S_i,S$ are demicontinuous, we first consider the continuity of the operators~$F_i$.
\begin{lemma}\label{lemma:Fconv}
If the \cref{ass:pstruct,ass:Gamma} hold, then the nonlinear operators $F_i:\Lambda_i\to V_i$ are continuous.
\end{lemma}

\begin{proof}
Let $\eta,\mu$ be elements in $\Lambda_i$. Using the equality~\cref{eq:trick} together with \cref{lemma:aiprop} gives us the bound 
\begin{equation}\label{eq:contbound1}
\begin{aligned}
	        c_i&(\|\nabla (F_i\eta-F_i\mu)\|^p_{L^p(\Omega_i)^d}+\|F_i\eta-F_i\mu\|^r_{L^r(\Omega_i)})\\
	        &\leq a_i(F_i\eta, F_i\eta-F_i\mu)-a_i(F_i\mu, F_i\eta-F_i\mu)\\
	        &=a_i\bigl(F_i\eta, R_i(\eta-\mu)\bigr)-a_i(F_i\eta,w_i)-a_i\bigl(F_i\mu, R_i(\eta-\mu)\bigr)+a_i(F_i\mu,w_i)\\
	        &\leq C_i \bigl(\|\nabla F_i\eta\|^{p-1}_{L^p(\Omega_i)^d}\|\nabla R_i(\eta-\mu)\|^{\phantom{p-1}}_{L^p(\Omega_i)^d}
	      			+\|F_i\eta\|_{L^r(\Omega_i)}^{r-1}\|R_i(\eta-\mu)\|^{\phantom{r-1}}_{L^r(\Omega_i)}\\
	        &\quad+\|\nabla F_i\mu\|^{p-1}_{L^p(\Omega_i)^d}\|\nabla R_i(\eta-\mu)\|^{\phantom{p-1}}_{L^p(\Omega_i)^d}
	        	 		+\|F_i\mu\|_{L^r(\Omega_i)}^{r-1}\|R_i(\eta-\mu)\|^{\phantom{r-1}}_{L^r(\Omega_i)}\bigr) \\
	        &\leq C_i \big(\|\nabla F_i\eta\|^{p-1}_{L^p(\Omega_i)^d}+\|F_i\eta\|_{L^r(\Omega_i)}^{r-1}\\
	        &\quad+\|\nabla F_i\mu\|^{p-1}_{L^p(\Omega_i)^d}+\|F_i\mu\|_{L^r(\Omega_i)}^{r-1}\big)\|\eta-\mu\|_{\Lambda_i}.
\end{aligned}
\end{equation}
Letting $\mu=0$ in~\cref{eq:contbound1} and employing the inequality $|x|^p-2^{p-1}|y|^p\leq 2^{p-1}|x-y|^p$, twice, yield that
\begin{align*}
	        c_i &(\|\nabla F_i\eta\|^p_{L^p(\Omega_i)^d}+\|F_i\eta\|^r_{L^r(\Omega_i)}-2^{p-1}\|\nabla F_i0\|^p_{L^p(\Omega_i)^d}
	        			-2^{r-1}\|F_i0\|^r_{L^r(\Omega_i)})\\
	        &\leq c_i(2^{p-1}\|\nabla (F_i\eta-F_i0)\|^p_{L^p(\Omega_i)^d}+2^{r-1}\|F_i\eta-F_i0\|^r_{L^r(\Omega_i)})\\
	        &\leq C_i \big(\|\nabla F_i\eta\|^{p-1}_{L^p(\Omega_i)^d}+\|F_i\eta\|_{L^r(\Omega_i)}^{r-1}
	         +\|\nabla F_i0\|^{p-1}_{L^p(\Omega_i)^d}+\|F_i0\|_{L^r(\Omega_i)}^{r-1}\big)\|\eta\|_{\Lambda_i}.
\end{align*}
Thus, we have a bound of the form
\begin{equation}\label{eq:contbound2}
\frac{\|\nabla F_i\eta\|^p_{L^p(\Omega_i)^d}+\|F_i\eta\|^r_{L^r(\Omega_i)}-c_1}{\|\nabla F_i\eta\|^{p-1}_{L^p(\Omega_i)^d}+\|F_i\eta\|_{L^r(\Omega_i)}^{r-1}+c_2}\leq C_i\|\eta\|_{\Lambda_i},
\end{equation}
for every $\eta\in\Lambda_i$, where $c_\ell=c_\ell(\|\nabla F_i0\|_{L^p(\Omega_i)^d},\|F_i0\|_{L^r(\Omega_i)})\geq 0$. 

Assume that $\eta^k\to\eta$ in $\Lambda_i$. As $\eta^k$ is bounded in $\Lambda_i$, the bound~\cref{eq:contbound2} implies that 
$\nabla F_i\eta^k$ and $F_i\eta^k$ are bounded in $L^p(\Omega_i)^d$ and $L^r(\Omega_i)$, respectively. By setting $\mu=\eta^k$ in~\cref{eq:contbound1}, we finally obtain that $\nabla F_i\eta^k\to\nabla F_i\eta$ in in $L^p(\Omega_i)^d$ and $F_i\eta^k\to F_i\eta$ in $L^r(\Omega_i)$, i.e., $F_i\eta^k\to F_i\eta$ in $V_i$.
\end{proof}

\begin{lemma}\label{lem:demicont}
If the \cref{ass:pstruct,ass:Gamma} hold, then the operators $S_i: \Lambda_i\rightarrow \Lambda_i^*$ and $S:\Lambda\rightarrow\Lambda^*$ are demicontinuous, i.e., if $\eta_i^k\to\eta_i$ in $\Lambda_i$ and  $\eta^k\to\eta$ in $\Lambda$ then 
\begin{displaymath}
\langle S_i\eta^k_i-S_i\eta_i,\mu_i\rangle\to 0\quad\text{and}\quad\langle S\eta^k-S\eta,\mu\rangle\to 0,
\end{displaymath}
for all $\mu_i\in\Lambda_i$ and $\mu\in\Lambda$.
\end{lemma}

\begin{proof}
Assume that $\eta_i^k\to\eta_i$ in $\Lambda_i$. \cref{lemma:Fconv} then implies that $\nabla F_i\eta^k\to\nabla F_i\eta$ in $L^p(\Omega_i)^d$ and $F_i\eta^k\to F_i\eta$ in $L^r(\Omega_i)$. By the assumed continuity and boundedness of the functions $\alpha:z\mapsto\bigl(\alpha_1(z),\ldots,\alpha_d(z)\bigr)$ and $g$, we also have that the corresponding Nemyckii operators $\alpha_\ell:L^p(\Omega_i)^d\to L^{\nicefrac{p}{(p-1)}}(\Omega_i)$ and $g:L^r(\Omega_i)\to L^{\nicefrac{r}{(r-1)}}(\Omega_i)$ are continuous~\cite[Proposition~26.6]{zeidler}. The demicontinuity of $S_i$ then holds, as
\begin{align*}
|\langle S_i\eta^k_i-S_i\eta_i,\mu_i\rangle| 
&\leq\bigl(\,\sum_{\ell=1}^d  \|\alpha_\ell(\nabla F_i\eta_i^k)-\alpha_\ell(\nabla F_i\eta_i)\|_{L^{\nicefrac{p}{(p-1)}}(\Omega_i)}\\
&\quad+\|g(F_i\eta_i^k)-g(F_i\eta_i)\|_{ L^{\nicefrac{r}{(r-1)}}(\Omega_i)}\bigr)\|R_i\mu_i\|_{V_i},
\end{align*}
for every $\mu_i\in\Lambda_i$. The demicontinuity of $S_i$ directly implies the same property for $S$, as a convergent sequence $\{\eta^k\}$ in $\Lambda$ is also convergent in $\Lambda_1$ and $\Lambda_2.$
\end{proof}

\begin{theorem}\label{thm:Sinv}
If the \cref{ass:pstruct,ass:Gamma} hold, then the nonlinear Steklov--Poincar\'e operators $S_i: \Lambda_i\rightarrow \Lambda_i^*$ and $S:\Lambda\rightarrow\Lambda^*$ are bijective. 
\end{theorem}

\begin{proof}
The spaces $\Lambda_i$ and $\Lambda$ are real, reflexive Banach spaces and, by \cref{lemma:monotonicity,lemma:coercivity,lem:demicont}, the operators $S_i: \Lambda_i\rightarrow \Lambda_i^*$ and $S:\Lambda\rightarrow\Lambda^*$ are all strictly monotone, coercive and demicontinuous. With these properties, the Browder--Minty theorem; see, e.g.,~\cite[Theorem~26.A(a,c,f)]{zeidler}, implies that the operators are bijective. 
\end{proof}
The next corollary follows by the same argumentation as for the bijectivity of $S_i$.

\begin{corollary}\label{cor:biject}
If the \cref{ass:pstruct,ass:Gamma} hold, then the operators $sJ+S_i:\Lambda_i\to\Lambda_i^*$ are bijective, for every $s>0$.
\end{corollary}

\section{Existence and convergence of the Robin--Robin method}\label{sec:conv}

The weak form of the Peaceman--Rachford splitting~\cref{eq:PR} seems to be too general for a convergence analysis. To remedy this, we will restrict the domains of the operators $S_{i}, S$ such that the Steklov--Poincar\'e equation~\cref{eq:speq} and the Peaceman--Rachford splitting can be interpreted on $L^{2}(\Gamma)$ instead of on the dual spaces $\Lambda_i^*, \Lambda^*$. This comes at the cost of requiring more regularity of the weak solution and of the initial guess $\eta^0_2$. 

More precisely, we define the operators $\mathcal{S}_i:\dom(\mathcal{S}_i)\subseteq L^{2}(\Gamma)\to L^{2}(\Gamma)$ as
\begin{displaymath}
\dom(\mathcal{S}_i)=\{\mu\in\Lambda_i: S_i\mu\in L^2(\Gamma)^*\}\quad\text{and}\quad 
\mathcal{S}_i\mu=J^{-1}S_i\mu\text{ for }\mu\in \dom(\mathcal{S}_i).
\end{displaymath}
Analogously, we introduce $\mathcal{S}:\dom(\mathcal{S})\subseteq L^{2}(\Gamma)\to L^{2}(\Gamma)$ given by
\begin{displaymath}
\dom(\mathcal{S})=\{\mu\in\Lambda: S\mu\in L^2(\Gamma)^*\}\quad\text{and}\quad 
\mathcal{S}\mu=J^{-1}S\mu\text{ for }\mu\in \dom(\mathcal{S}).
\end{displaymath}
As the zero functional obviously is an element in $L^2(\Gamma)^*$, the unique solution $\eta\in\Lambda$ of the Steklov--Poincar\'e equation is in $\dom(\mathcal{S})$ and 
\begin{displaymath}
\mathcal{S}\eta=0.
\end{displaymath}

\begin{remark}
By the above construction, one obtains that $\dom(\mathcal{S}_1)\cap \dom(\mathcal{S}_2)\subseteq \dom(\mathcal{S})$ and 
\begin{displaymath}
\mathcal{S}\mu=\mathcal{S}_1\mu+\mathcal{S}_2\mu,\quad\text{ for all }\mu\in \dom(\mathcal{S}_1)\cap \dom(\mathcal{S}_2).
 \end{displaymath}
However, the definition of the domains do not ensure that $\dom(\mathcal{S})$ is equal to $\dom(\mathcal{S}_1)\cap \dom(\mathcal{S}_2)$, as $(S_{1}+S_{2})\mu\in L^2(\Gamma)^*$ does not necessarily imply that $S_{i}\mu\in L^2(\Gamma)^*$. 
\end{remark}

If the weak solution of the nonlinear elliptic equation~\cref{eq:weak} satisfies the additional regularity property stated in \cref{ass:regularity}, then the corresponding solution of $\mathcal{S}\eta=0$ is in fact an element in $\dom(\mathcal{S}_1)\cap \dom(\mathcal{S}_2)$. This propagation of regularity will be crucial when proving convergence of the Peaceman--Rachford splitting. 

\begin{lemma}\label{lemma:etareg}
If the \cref{ass:pstruct,ass:regularity,ass:Gamma} hold and $\mathcal{S}\eta=0$, then
$\eta\in \dom(\mathcal{S}_1)\cap \dom(\mathcal{S}_2)$.
\end{lemma}

\begin{proof}
As $u=\{F_1\eta\text{ on }\Omega_1;  F_2\eta\text{ on }\Omega_2\}$ is the weak solution of~\cref{eq:weak}, we have that 
\begin{displaymath}
\int_\Omega \alpha(\nabla u)\cdot\nabla v\,\mathrm{d}x= -\int_\Omega \bigl(g(u)-f\bigr)v\,\mathrm{d}x,\quad\text{for all }v\in C^\infty_{0}(\Omega). 
\end{displaymath}
The restrictions on $r$ in \cref{ass:pstruct} yield that $u\in W^{1,p}(\Omega)\hookrightarrow L^{2(r-1)}(\Omega)$. This together with the observation $|g(u)|^2\leq C |u|^{2(r-1)}$ implies that $g(u)-f\in L^2(\Omega)$, i.e., the distributional divergence of $\alpha(\nabla u)$ is in $ L^2(\Omega)^d$. By \cref{ass:regularity} and restricting to $\Omega_i$, we arrive at
\begin{gather*}
\alpha(\nabla F_i\eta)\in H(\mathrm{div},\Omega_i)\cap C(\overline{\Omega}_i)^d,\quad 
\alpha(\nabla F_i\eta)\cdot\nu_i\in L^\infty(\partial\Omega_i)\quad\text{and}\\
\quad\nabla\cdot\alpha(\nabla F_i\eta)=g(F_i\eta)-f_i\in L^2(\Omega_i).
\end{gather*}
The $H(\mathrm{div},\Omega_i)$-version of Green's formula~\cite[Chapter 1, Corollary 2.1]{raviart} then gives us
\begin{displaymath}
\int_{\Omega_i} \alpha(\nabla F_i\eta)\cdot\nabla v\,\mathrm{d}x =
						 -\int_{\Omega_i} \nabla\cdot\alpha(\nabla F_i\eta) v\,\mathrm{d}x
						 + \int_{\partial\Omega_i}\alpha(\nabla F_i\eta)\cdot\nu_i\, T_{\partial\Omega_i}v\,\mathrm{d}S,
\end{displaymath}
for all $v\in H^1(\Omega_i)$. Hence, 
\begin{align*}
\langle S_i\eta, \mu\rangle &= \int_{\Omega_i} \alpha(\nabla F_i\eta)\cdot\nabla R_i\mu\,\mathrm{d}x 
				         						+ \int_{\Omega_i} \bigl(g(F_i\eta)-f_i\bigr)R_i\mu\,\mathrm{d}x\\
					& = \int_{\partial\Omega_i}\alpha(\nabla F_i\eta)\cdot\nu_i\, T_{\partial\Omega_i}R_i\mu\,\mathrm{d}S
					   = \bigl(\alpha(\nabla F_i\eta)\cdot\nu_i,\mu\bigr)_{L^2(\Gamma)},\quad\text{for all }\mu\in\Lambda_i,
\end{align*}
which implies that $\eta\in \dom(\mathcal{S}_i)$, for $i=1,2$. 
\end{proof}

\begin{remark}\label{rem:reg}
From the proof it is clear that the regularity assumption $\alpha(\nabla u)\in C(\overline{\Omega})^d$ is stricter than necessary, and could be replaced by assuming that the normal component of $\alpha(\nabla u)$ on $\Gamma$ can be interpreted as an element in $L^2(\Gamma)^*$. However, characterizing the spatial regularity of $u$ required to satisfy this weaker assumption demands a more elaborate trace theory than the one considered in \cref{sec:functionspaces}.
\end{remark}

\begin{lemma}\label{lemma:Smonbij}
If the \cref{ass:pstruct,ass:Gamma} hold, then the operators $\mathcal{S}_{i}$ are monotone, i.e., 
\begin{displaymath}
(\mathcal{S}_{i}\eta-\mathcal{S}_{i}\mu,\eta-\mu)_{L^2(\Gamma)}\geq 0,\quad\text{for all }\eta,\mu\in \dom(\mathcal{S}_i),
\end{displaymath}
and the operators $sI+\mathcal{S}_i:\dom(\mathcal{S}_i)\to L^{2}(\Gamma)$ are bijective for any $s>0$.
\end{lemma}

\begin{proof}
The monotonicity follows by \cref{lemma:monotonicity}, as 
\begin{align*}
(\mathcal{S}_{i}\eta-\mathcal{S}_{i}\mu,\eta-\mu)_{L^2(\Gamma)} &= (J^{-1}S_{i}\eta-J^{-1}S_{i}\mu,\eta-\mu)_{L^2(\Gamma)}\\
				& =\langle S_i\eta-S_i\mu, \eta-\mu \rangle \geq 0,\quad\text{for all }\eta,\mu\in \dom(\mathcal{S}_i)\subseteq\Lambda_{i}.
\end{align*}
For a fixed $s>0$ and an arbitrary $\mu\in L^2(\Gamma)$ we have, due to \cref{cor:biject}, that there exists a unique $\eta\in\Lambda_i$ such that $(sJ+S_{i})\eta=J\mu$ in $\Lambda_{i}^{*}$, i.e., 
\begin{displaymath}
S_{i}\eta=J(\mu-s\eta)\in L^2(\Gamma)^*. 
\end{displaymath}
Hence, $\eta\in \dom(\mathcal{S}_{i})$ and $(sI +\mathcal{S}_{i})\eta=\mu$ in $L^2(\Gamma)$. The operators $sI+\mathcal{S}_i:\dom(\mathcal{S}_i)\to L^{2}(\Gamma)$ are therefore bijective. 
\end{proof}

The Peaceman--Rachford splitting on $L^2(\Gamma)$ is now given by finding $(\eta^{n}_{1}, \eta^{n}_{2})\in \dom(\mathcal{S}_1)\times \dom(\mathcal{S}_2)$, for $n=1,2,\ldots,$ such that 
\begin{equation}\label{eq:PRL2}
\left\{\begin{aligned}
 	(sI+\mathcal{S}_1)\eta_1^{n+1}&=(sI-\mathcal{S}_2)\eta_2^n, \\
	(sI+\mathcal{S}_2)\eta_2^{n+1}&=(sI-\mathcal{S}_1)\eta_1^{n+1},
\end{aligned}\right.
\end{equation}
where $\eta_2^0\in \dom(\mathcal{S}_2)$ is an initial guess. \cref{lemma:Smonbij} then directly yields the existence of the approximation. 

\begin{corollary}\label{cor:PRL2exist}
If the \cref{ass:pstruct,ass:Gamma} hold and  $\eta_2^0\in \dom(\mathcal{S}_2)$, then there exists a unique Peaceman--Rachford approximation $(\eta^{n}_{1}, \eta^{n}_{2})_{n\geq 1}\subset \dom(\mathcal{S}_1)\times \dom(\mathcal{S}_2)$ given by~\cref{eq:PRL2} in $L^2(\Gamma)$.
\end{corollary}

\begin{corollary}\label{cor:PRisRobin} 
Let the \cref{ass:pstruct,ass:Gamma} hold, $\eta_2^0\in \dom(\mathcal{S}_2)$ and set $u^0_2=F_2\eta_2^0$. The Peaceman--Rachford approximation $(\eta^{n}_{1}, \eta^{n}_{2})_{n\geq 1}\subset \dom(\mathcal{S}_1)\times \dom(\mathcal{S}_2)$ also satisfies the weak formulation~\cref{eq:PR}, and $(u^{n}_1, u^{n}_2)_{n\geq 1}=(F_1\eta^{n}_{1}, F_2\eta^{n}_{2})_{n\geq 1}$ is a weak Robin--Robin approximation~\cref{eq:Robin}.
\end{corollary}

\begin{proof}
Assume that $(\eta^{n}_{1}, \eta^{n}_{2})_{n\geq 1}\subset \dom(\mathcal{S}_1)\times \dom(\mathcal{S}_2)$ is a Peaceman--Rachford approximation in  $L^2(\Gamma)$. Then,
\begin{align*}
\bigl((sI+\mathcal{S}_1)\eta_1^{n+1},\mu\bigr)_{L^2(\Gamma)}&=\bigl((sI-\mathcal{S}_2)\eta_2^n,\mu\bigr)_{L^2(\Gamma)}, &  &\text{for all }\mu\in L^2(\Gamma),\\
\bigl((sI+\mathcal{S}_1)\eta_1^{n+1},\mu\bigr)_{L^2(\Gamma)}&=\langle(sJ+S_1)\eta_1^{n+1}, \mu\rangle,& &\text{for all }\mu\in\Lambda_1,\quad
\text{and}\\
\bigl((sI-\mathcal{S}_2)\eta_2^n,\mu\bigr)_{L^2(\Gamma)}&=\langle (sJ-S_2)\eta_2^n, \mu\rangle, & &\text{for all }\mu\in\Lambda_2.
\end{align*}
This implies that 
\begin{displaymath}
\langle(sJ+S_1)\eta_1^{n+1}, \mu\rangle=\langle (sJ-S_2)\eta_2^n, \mu\rangle,\quad\text{for all }\mu\in\Lambda=\Lambda_1\cap\Lambda_2,
\end{displaymath}
i.e., the first assertion of~\cref{eq:PR} holds. The same argumentation yields that the second assertion of~\cref{eq:PR} is valid. As $(\eta^{n}_{1}, \eta^{n}_{2})_{n\geq 1}$ satisfies~\cref{eq:PR}, \cref{lem:RobinPRequiv} directly implies that $(u^{n}_1, u^{n}_2)_{n\geq 1}=(F_1\eta^{n}_{1}, F_2\eta^{n}_{2})_{n\geq 1}$ is a weak Robin--Robin approximation~\cref{eq:Robin}.
\end{proof}

\begin{remark}
At a first glance, finding an initial guess satisfying $\eta^0_2\in \dom(\mathcal{S}_2)$ might seem limiting, as the domain is not explicitly given. However, such an initial guess can, e.g., be found by solving $\langle S_2\eta_2^0,\mu\rangle=0$, for all $\mu\in\Lambda_2$.
\end{remark}

With this $L^2(\Gamma)$-framework the key part of the convergence proof follows by the abstract result~\cite[Proposition~1]{lionsmercier}. 
For sake of completeness we state a simplified version of the short proof in the current notation. 

\begin{lemma}\label{lemma:LM}
Consider the solution of $\mathcal{S}\eta=0$ and the Peaceman--Rachford approximation $(\eta^{n}_{1}, \eta^{n}_{2})_{n\geq 1}$. 
If $\eta_2^0\in \dom(\mathcal{S}_2)$ and the \cref{ass:pstruct,ass:regularity,ass:Gamma} hold, then
\begin{equation}\label{eq:Sconv}
(\mathcal{S}_i\eta^n_i-\mathcal{S}_i \eta, \eta^n_i -\eta)_{L^2(\Gamma)}\to 0,\quad\text{as }n\to\infty,
\end{equation}
for $i=1, 2$.
\end{lemma}

\begin{proof}
By the hypotheses and \cref{lemma:etareg}, we obtain that $\eta\in \dom(\mathcal{S}_1)\cap \dom(\mathcal{S}_2)$ and $\mathcal{S}_1\eta=-\mathcal{S}_2\eta$. Furthermore, 
\begin{displaymath}
\eta^{n+1}_1=(sI+\mathcal{S}_1)^{-1}(sI-\mathcal{S}_2)\eta^n_2\in \dom(\mathcal{S}_1)
\text{ and } 
\eta^{n+1}_2=(sI+\mathcal{S}_2)^{-1}(sI-\mathcal{S}_1)\eta^{n+1}_1\in \dom(\mathcal{S}_2).
\end{displaymath}
Next, we introduce the notation 
\begin{displaymath}
\mu^n=(sI+\mathcal{S}_2)\eta^n_2,\quad \mu=(sI+\mathcal{S}_2)\eta,\quad 
\lambda^n=(sI-\mathcal{S}_2)\eta^n_2\quad\text{and}\quad\lambda=(sI-\mathcal{S}_2)\eta,
\end{displaymath}
which yields the representations
\begin{align*}
\eta &=\frac{\mu+\lambda}{2s},  &    \mathcal{S}_2\eta&=\frac{\mu-\lambda}{2},  & &\mathcal{S}_1\eta=\frac{\lambda-\mu}{2},  \\[2pt]
\eta^n_2 &=\frac{\mu^n+\lambda^n}{2s},  &    \mathcal{S}_2\eta^n_2&=\frac{\mu^n-\lambda^n}{2},   & &\\[2pt]
\eta^{n+1}_1 &=\frac{\mu^{n+1}+\lambda^n}{2s},  &   \mathcal{S}_1\eta^{n+1}_1&=\frac{\lambda^n-\mu^{n+1}}{2}.  & &
\end{align*}
The monotonicity of $\mathcal{S}_i$ then gives the bounds 
\begin{align*}
0 &\leq (\mathcal{S}_2\eta^n_2-\mathcal{S}_2 \eta, \eta^n_2 -\eta)_{L^2(\Gamma)}\\ 
&= \frac{1}{4s} \bigl((\mu^n-\mu)-(\lambda^n-\lambda),(\mu^n-\mu)+(\lambda^n-\lambda)\bigr)_{L^2(\Gamma)} \\
&=\frac{1}{4s} \bigl(\|\mu^n-\mu\|^2_{L^2(\Gamma)}-\|\lambda^n-\lambda\|^2_{L^2(\Gamma)}\bigr),
\end{align*}
and
\begin{align*}
0 &\leq (\mathcal{S}_1\eta^{n+1}_1-\mathcal{S}_1 \eta, \eta^{n+1}_1 -\eta)_{L^2(\Gamma)}\\ 
&= \frac{1}{4s} \bigl((\lambda^n-\lambda)-(\mu^{n+1}-\mu),(\lambda^n-\lambda)+(\mu^{n+1}-\mu)\bigr)_{L^2(\Gamma)} \\
&=\frac{1}{4s} \bigl(\|\lambda^n-\lambda\|^2_{L^2(\Gamma)}-\|\mu^{n+1}-\mu\|^2_{L^2(\Gamma)}\bigr).
\end{align*}
Putting this together yields that 
\begin{displaymath}
\|\mu^{n+1}-\mu\|^2_{L^2(\Gamma)}\leq \|\lambda^n-\lambda\|^2_{L^2(\Gamma)} \leq \|\mu^n-\mu\|^2_{L^2(\Gamma)},
\end{displaymath}
and we obtain the telescopic sum
\begin{displaymath}
0\leq \sum_{n=0}^N \bigl(\|\mu^n-\mu\|^2_{L^2(\Gamma)}-\|\mu^{n+1}-\mu\|^2_{L^2(\Gamma)}\bigr)\leq \|\mu^0-\mu\|^2_{L^2(\Gamma)}- \|\mu^{N+1}-\mu\|^2_{L^2(\Gamma)},
\end{displaymath}
i.e., $\|\mu^n-\mu\|^2_{L^2(\Gamma)}-\|\mu^{n+1}-\mu\|^2_{L^2(\Gamma)}\to 0$ as $n\to \infty$. The latter together with the bounds above imply  the sought after limits~\cref{eq:Sconv}.
\end{proof}

\begin{theorem}\label{thm:conv}
Consider the Peaceman--Rachford approximation $(\eta^{n}_{1}, \eta^{n}_{2})_{n\geq 1}$, given by~\cref{eq:PRL2}, of the Steklov--Poincaré equation $\mathcal{S}\eta=0$ in $L^2(\Gamma)$, together with the corresponding Robin--Robin approximation $(u^{n}_1, u^{n}_2)_{n\geq 1}=(F_1\eta^{n}_{1}, F_2\eta^{n}_{2})_{n\geq 1}$ of the weak solution $u=\{F_1\eta\text{ on }\Omega_1;  F_2\eta\text{ on }\Omega_2\}$ to the nonlinear elliptic equation~\cref{eq:weak}.\\
If $\eta_2^0\in \dom(\mathcal{S}_2)$ and the \cref{ass:pstruct,ass:regularity,ass:Gamma} hold, then
\begin{equation}\label{eq:conv}
\|\eta^n_1-\eta\|_{\Lambda_1}+\|\eta^n_2-\eta\|_{\Lambda_2}\to 0, \quad\text{and}\quad
\|u^n_1-u\|_{W^{1,p}(\Omega_1)}+\|u^n_2-u\|_{W^{1,p}(\Omega_2)}\to 0,
\end{equation}
as $n$ tends to infinity. 
\end{theorem}

\begin{proof}
By the monotonicity bound in \cref{lemma:monotonicity}, the property that $\eta^n_i, \eta\in \dom(\mathcal{S}_i)$ and \cref{lemma:LM}, we have the limits 
\begin{align*}
c_i\bigl(\|\nabla(F_i \eta^n_i-F_i \eta)\|_{L^p(\Omega_i)^d}^p &+\|F_i \eta^n_i-F_i \eta\|_{L^r(\Omega_i)}^r\bigr)
					\leq  \langle S_i\eta^n_i-S_i\eta, \eta^n_i-\eta \rangle\\ &= (\mathcal{S}_i\eta^n_i-\mathcal{S}_i \eta, \eta^n_i -\eta)_{L^2(\Gamma)}\to 0,\quad\text{as }n\to\infty,
\end{align*}
for $i=1,2$. Hence, each of the terms $\|\nabla(F_i \eta^n_i-F_i \eta)\|_{L^p(\Omega_i)^d}$ and $\|F_i \eta^n_i-F_i \eta\|_{L^r(\Omega_i)}$ tend to zero, which  yields that
\begin{displaymath}
\|\eta^n_i-\eta\|_{\Lambda_i} \leq C \|F_i \eta^n_i-F_i \eta\|_{V_i}\to 0,\quad\text{as }n\to\infty,
\end{displaymath}
for $i=1,2$. The desired convergence~\cref{eq:conv} is then proven, as $\|\cdot\|_{V_i}$ and $\|\cdot\|_{W^{1,p}(\Omega_i)}$ are equivalent norms. 
\end{proof}

\bibliographystyle{siamplain}
\bibliography{references}
\end{document}